\documentclass[twoside]{amsart}
\usepackage[dvips]{graphicx}   
\usepackage{color,epsfig}      

\usepackage{amsmath,amssymb,enumerate}
\usepackage[margin=2.5cm,nohead]{geometry}
\usepackage[utf8]{inputenc}
\usepackage[english]{babel}
\usepackage{latexsym}
\usepackage{amsfonts}
\usepackage{amssymb}
\usepackage{amsthm}
\usepackage{setspace}
\usepackage{graphicx}
\usepackage{epstopdf}
\usepackage{caption}
\usepackage{lmodern}
\usepackage{indentfirst}
\usepackage{geometry}
\usepackage{faktor}

\newtheorem{theorem}{Theorem}[section]

\newtheorem{proposition}[theorem]{Proposition}
\newtheorem{lemma}[theorem]{Lemma}
\newtheorem{remark}[theorem]{Remark}

\newtheorem{example}[theorem]{Example}

\newcommand{\Diff}{\mathrm{Diff}}
\newcommand{\tb}{\mathrm{tb}}
\renewcommand{\sl}{\mathrm{sl}}

\newcommand{\std}{\mathrm{std}}
\newcommand{\rot}{\mathrm{rot}}
\newcommand{\w}{\mathrm{w}}

\newcommand{\HFK}{\mathrm{HFK}}
\newcommand{\id}{\mathrm{id}}
\newcommand{\Sat}{\mathrm{Sat}}

\begin{document}

\title[] {Legendrian Non-Simple Two-Bridge Knots}

\author[Vikt\'oria F\"oldv\'ari]{Vikt\'oria F\"oldv\'ari}

\address{Vikt\'oria F\"oldv\'ari,
Institute of Mathematics, E\"otv\"os Lor\'and University, Budapest, Hungary
}
\email{foldvari@cs.elte.hu}
                                 
\keywords{Legendrian knots, transverse knots, two-bridge knots, knot Floer homology}

                                                                                
\begin{abstract}
  By examining knot Floer homology, we extend a result of Ozsváth and Stipsicz and show further infinitely many Legendrian and transversely non-simple knot types among two-bridge knots. We give sufficient conditions of Legendrian and transverse non-simplicity on the continued fraction expansion of the corresponding rational number.
\end{abstract}
\maketitle

\section{Introduction}
A \emph{Legendrian knot} in a contact 3-manifold $(M^3, \xi)$ is an embedded circle that is everywhere tangent to $\xi$.
There are three classical invariants of Legendrian knots: the \emph{smooth knot type}, the \emph{Thurston-Bennequin number} ($\tb$) and the \emph{rotation number} ($\rot$).
A natural question is if these data are enough to classify Legendrian knots up to (Legendrian) isotopy. It turned out that the general answer is no, but there exist knots whose $\tb$ and $\rot$ determine the Legendrian knot type. We say that a knot type is \emph{Legendrian simple} if any two Legendrian realizations of it with equal Thurston-Bennequin and rotation number are Legendrian isotopic.
\\
An embedded circle in a contact 3-manifold $(M^3, \xi)$ that is everywhere transverse to $\xi$ is called a \emph{transverse knot}. Besides the smooth knot type, the classical invariant of a transverse knot is the \emph{self-linking number} ($\sl$). Similarly to Legendrian simplicity, we call a knot type \emph{transversely simple} if any two transverse realizations with equal self-linking number are transverse isotopic.

The unknot \cite{E93} and all torus knots \cite{EH01} are both Legendrian and transversely simple.
The first examples for Legendrian non-simple knots were found by Chekanov \cite{Ch02}. Since then, more examples were detected, see \cite{EFM01, Ng05}. However, we still cannot categorize Legendrian and transverse knots up to isotopy. For twist knots, Etnyre, Ng and Vértesi \cite{ENgV13} gave a complete classification. Ozsváth and Stipsicz \cite{OS10} examined two-bridge knots and gave sufficient conditions for transverse non-simplicity on the corresponding rational number.  

In this paper we determine further Legendrian and transversely non-simple two-bridge knots. The family of \emph{two-bridge knots} consists of knots which admit a diagram such that the natural height function has exactly four critical points: two maxima and two minima. We also call them \emph{rational knots}, since they can be classified with rational numbers. If not stated otherwise, we consider Legendrian and transverse knots in $(S^3, \xi_{\std})$, where $\xi_{\std}$ is the standard contact structure on $S^3$.

By the theorem of Epstein, Fuchs and Meyer \cite{EFM01} we know that every transverse knot type can be realized as the transverse push-off of some Legendrian knot type. Moreover, two oriented Legendrian knots become Legendrian isotopic after some number of negative stabilizations if and only if their transverse push-offs are transversely isotopic.
That is, Legendrian simplicity implies transverse simplicity. However, the converse is not true: there are knot types which are transversely simple but Legendrian non-simple, for examples see \cite{EFM01}.

In Theorem \ref{sajat} we show Legendrian non-simplicity for some two-bridge knots.

\begin{theorem}\label{sajat}
Suppose that $\frac{p}{q}\in \mathbb{Q}$ has the continued fraction expansion $\frac{p}{q}=[a_1,...,a_{2m+1}]$ for $m\geq 1$ and suppose that 
\begin{itemize}
\item $a_2$ is odd,
\item $a_1$ and $a_{2i}$ is even for $i>1$,
\item $\sum\limits_{i=1}^m a_{2i+1}$ is odd.
\end{itemize}
Then the corresponding 2-bridge knot $K$ admits at least $\lceil \frac{a_1}{4} \rceil$ distinct Legendrian realizations with $\tb=\sum\limits_{i=1}^m a_{2i+1}$ and $\rot=0$.
\end{theorem}

The proof uses the Legendrian invariant $\mathcal{L}$, introduced in \cite{LOSSz08}. It is known that $\mathcal{L}$ is invariant under negative stabilization, therefore, it is also an invariant of the transverse knot type. Due to these, Theorem \ref{sajat} and its proof is also true for transverse knots:

\begin{theorem}\label{sajattransv}
Suppose that $\frac{p}{q}\in \mathbb{Q}$ has the continued fraction expansion $\frac{p}{q}=[a_1,...,a_{2m+1}]$ for $m\geq 1$ and suppose that 
\begin{itemize}
\item $a_2$ is odd,
\item $a_1$ and $a_{2i}$ is even for $i>1$,
\item $\sum\limits_{i=1}^m a_{2i+1}$ is odd.
\end{itemize}
Then the corresponding 2-bridge knot $K$ admits at least $\lceil \frac{a_1}{4} \rceil$ distinct transverse realizations with $\sl=\sum\limits_{i=1}^m a_{2i+1}$.
\end{theorem}

Ozsváth and Stipsicz already stated a very similar result \cite[Theorem 5.8.]{OS10}. Our theorem is a generalization of theirs, in the sense that they assumed that all $a_i$ for $i\neq 2,3$ are even, while we drop the conditions on the terms $a_{2i+1}$ for $i>1$. This way we get much more new examples of Legendrian and transversely non-simple two-bridge knots.

However, the two-bridge knots mentioned in Theorem \ref{sajat} and \ref{sajattransv} are still not all the Legendrian and transversely non-simple ones. In Section \ref{sec:alg}, we show even more examples of Legendrian non-simple two-bridge knots that are not included in the above theorems. The way we verify their non-simplicity is a generalization of the proof of Theorem \ref{sajat}. As a demonstration, we will prove the following:
\begin{proposition}\label{prop:computation}
The two-bridge knot $K_\frac{14}{1825}$ is Legendrian non-simple.
\end{proposition}
Furthermore, we give a method to decide whether an arbitrary two-bridge knot can be proven to be Legendrian non-simple using the same tools as in Theorem \ref{sajat} and \ref{sajattransv}.

$\mathbf{Acknowledgement:}$ I thank my supervisor, András Stipsicz, for his help and guidance.

\section{Proof of Theorem \ref{sajat} and \ref{sajattransv}}\label{bizonyitasok}



We prove Theorem \ref{sajat}; the proof of Theorem \ref{sajattransv} is the same. We follow the proof of Ozsváth and Stipsicz for \cite[Theorem 5.8.]{OS10}.

\begin{proof}[Proof of Theorem \ref{sajat}]
Let $L$ be a Legendrian knot in a contact manifold and $\widetilde{L}\subset (S^1 \times D^2, \xi_{\std}^T)$ a Legendrian link in the solid torus with the standard contact structure on it. Consider a tubular neighbourhood $\nu L$ of $L$. By the tubular neighbourhood theorem \cite{L98}, there is a contactomorphism between $\nu L$ and the solid torus containing $\widetilde{L}$.
Therefore, we can embed $\widetilde{L}$ into $\nu L$. The new Legendrian knot obained this way is $\Sat(L,\widetilde{L})$, the \emph{Legendrian satellite} of $L$ and $\widetilde{L}$.

For the Thurston-Bennequin number and the rotation number the following equations hold \cite{OS10}:
$$\tb(\Sat(L,\widetilde{L}))=\w(\widetilde{L})^2\cdot \tb(L)+\tb(\widetilde{L}),$$
$$\rot(\Sat(L,\widetilde{L}))=\w(\widetilde{L})\cdot \rot(L)+\rot(\widetilde{L}),$$
where $\w(\widetilde{L})$ denotes the \emph{winding number} of $\widetilde{L}$, that is, the number of times $\widetilde{L}$ goes around the $S^1$ direction of the solid torus $S^1 \times D^2$. Note that if $\w(\widetilde{L})=0$ then the invariants tb and rot of the satellite knot are independent of $L$.

\begin{figure}[ht]
\centering
\includegraphics[width=0.55\textwidth]{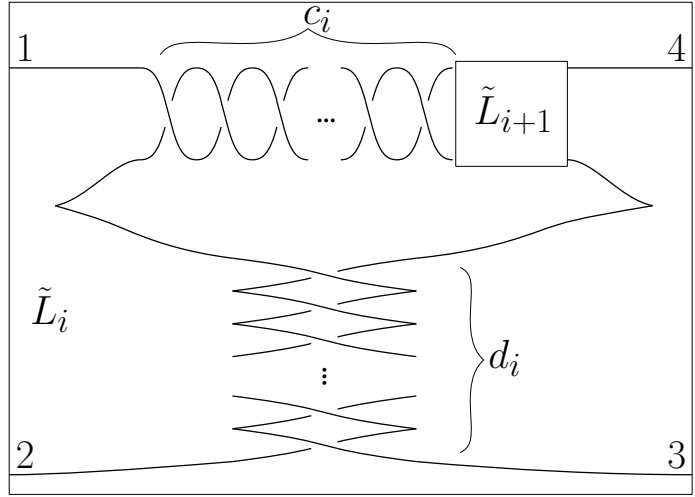}
\caption{The solid torus knot $\widetilde{L}_i$. The data $c_i$ and $d_i$ for $i=1,...,m$ determine $\widetilde{L}$.}
\label{fig:Lihullam}
\end{figure}

Let $\widetilde{L}_i$ be the solid torus knot (i.e. a knot in $S^1\times D^2$) with $c_i$ and $d_i$ crossings given in Figure \ref{fig:Lihullam}. (We will refer to the numbering of the arcs shown in the figure later.) Starting with $i=2$ we put $\widetilde{L}_i$ into the box of $\widetilde{L}_{i-1}$. We choose $\widetilde{L}$ to be the knot obtained by putting together $m$ pieces this way ($a_{2i}=d_i$, $a_{2i+1}=c_i$ for $i=1,...,m$). (Note that $\widetilde{L}$ is independent from $a_1$.) For $i=m$, the innermost element $\widetilde{L}_m$ contains no further terms in its inner box, inside that there are two horizontal arcs instead. Also note that for the outermost term $\widetilde{L}_1$ the left and right hand side of the rectangle shown in Figure \ref{fig:Lihullam} are identified, thus, the arcs numbered with 1 and 4 and the arcs 2 and 3 are connected.

\begin{lemma}\label{lemma:w=0}
$\w({\widetilde{L}})=0$ in the case when
\begin{itemize}
\item $a_2$ is odd,
\item $a_1$ and $ a_{2i}$ is even for $i>1$ ,
\item $\sum\limits_{i=1}^m a_{2i+1}$ is odd.
\end{itemize}
\end{lemma}

The proof of Lemma \ref{lemma:w=0} will follow from the general method we give in Section \ref{sec:alg}.

Consider $\widetilde{L}$ as constructed above. $\widetilde{L}$ is determined by $a_2$, $a_3$,...,$a_{2m+1}$. Let $k$ and $l$ be odd numbers such that $k+l=a_1$. For $k=2a+1$ and $l=2b+1$, let $U(a,b)$ denote a Legendrian realization of the unknot with $\tb(U(a,b))=-\frac{k+l}{2}$ and $\rot(U(a,b))=k-l$ shown in Figure \ref{fig:unknot}, and consider the Legendrian satellite $L_{k,l}=\Sat(U(a,b),\widetilde{L})$. $L_{k,l}$ is smoothly isotopic to the 2-bridge knot $K_{\frac{p}{q}}=[a_1,...,a_{2m+1}]$ where $\frac{p}{q}=[a_1,...,a_{2m+1}]$ is the continued fraction expansion of $\frac{p}{q}$.

\begin{figure}[ht]
\centering
\includegraphics[width=0.4\textwidth]{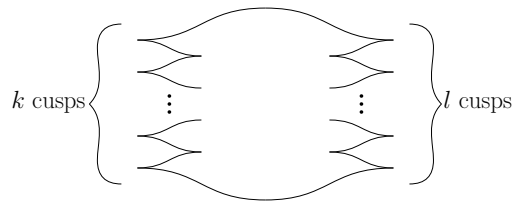}
\caption{$U(a,b)$, a Legendrian realization of the unknot}
\label{fig:unknot}
\end{figure}

Consider the 2-component link that consists of $L_{k,l}$ and an unknot so that the underlying smooth figure looks like Figure \ref{fig:link}. The linking number of the two knots is zero, since $\w({\widetilde{L}})=0$. After doing contact $(-1)$-surgery along the unknot component we get the Legendrian knot $L^\prime_{k,l}$ in the lens space $L(a+b+2,1)$ with contact structure $\xi_{k,l}$. 

\begin{figure}[ht]
\centering
\includegraphics[width=0.3\textwidth]{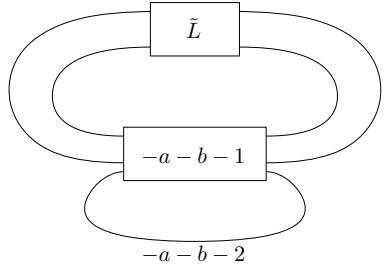}
\caption{The smooth knot $L_{k,l}$ with an unknot}
\label{fig:link}
\end{figure}

Let $U'$ denote the Legendrian push-off of the unknot component. Applying contact $(+1)$-surgery along $U'$ cancels the first contact $(-1)$-surgery and gives back the standard contact 3-sphere with $L_{k,l}$. Consider the distinguished triangle of knots induced by this contact $(+1)$-surgery along $U'$ (see Figure \ref{fig:triangle} in $-S^3$ with mirrors). According to Figure \ref{fig:triangle}, the third term of the triangle is a 2-component link: its first component is a knot denoted by $K_0$ and its second component is an unknot along which we do a surgery. Since this surgery curve is unlinked to $K_0$, $K_0$ is null-homologous.

\begin{figure}[ht]
\centering
\includegraphics[width=1\textwidth]{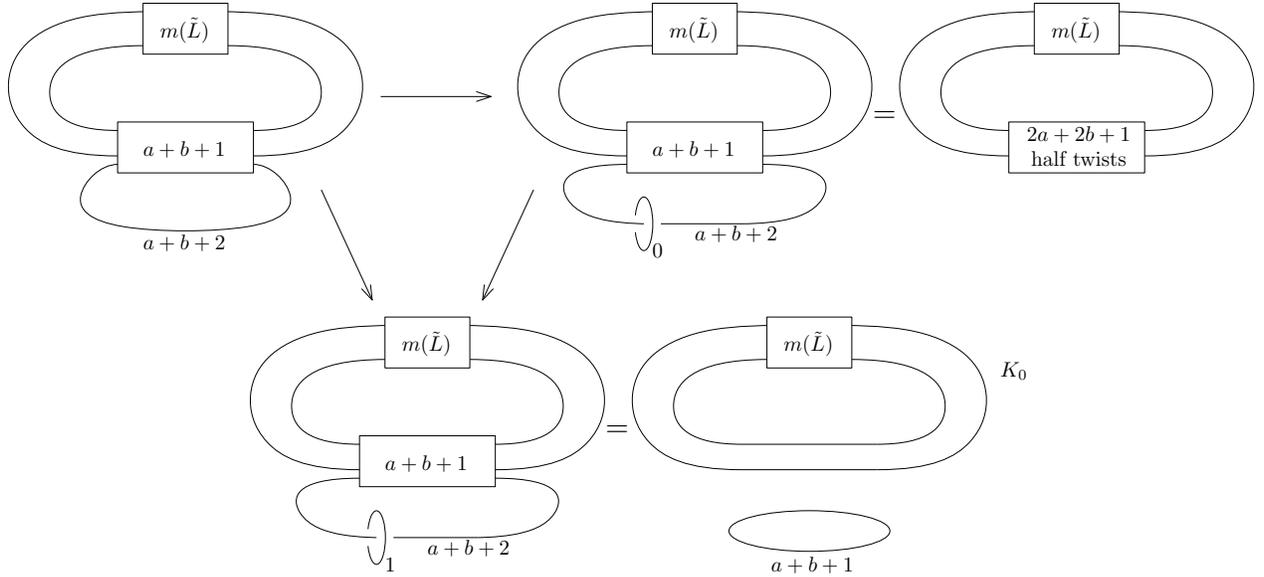}
\caption{The distinguished triangle of the contact (+1)-surgery\\
Notice that $m({L^\prime_{k,l}})\subset L(a+b+2,1)$; $m({L_{k,l}})\subset -S^3$ and
$K_0\subset L(a+b+1,1)$.}
\label{fig:triangle}
\end{figure}

This distinguished triangle of knots induces an exact triangle of knot Floer homologies $\widehat{\HFK}$, see \cite[Theorem 8.2.]{OSz04}. Contact $(+1)$-surgery along $S$ induces the following map on the homologies:
$$\widehat{F}_S: \widehat{\HFK}(-L(a+b+2,1), L^\prime_{k,l})\rightarrow \widehat{\HFK}(-S^3, L_{k,l}).$$
This induced map preserves the Alexander grading \cite{OSz04}.
Let $\tilde{A}=A(\mathcal{L}(L_{k,l}))$ denote the Alexander grading of the Legendrian invariant $\mathcal{L}(L_{k,l})$.

\begin{lemma}\label{lemma:HFKA=0}
In Alexander grading $\tilde{A}=A(\mathcal{L}(L_{k,l}))$ the $\widehat{\HFK}$ homology of the third term $(-L(a+b+1,1), K_0)$ vanishes when the conditions of Theorem \ref{sajat} hold.
\end{lemma}

The proof of Lemma \ref{lemma:HFKA=0} can be found in Section \ref{sec:alg}.

Then, in the exact triangle of homologies in this grading $\widehat{F}_S$ is an isomorphism. 

Now we are showing that for different $(k,l)$ pairs the Legendrian invariants $\mathcal{L}(L^\prime_{k,l})$ are different:

Consider $$\HFK^-(-L(a+b+2,1), L^\prime_{k,l}) \stackrel{U=1}{\longrightarrow} \mathrm{\widehat{HF}}(-L(a+b+2,1))$$
$$\mathcal{L}(L^\prime_{k,l})\mapsto \mathrm{c}(L(a+b+2,1), \xi_{k,l}).$$

For different $(k,l)$ pairs the rotation of the unknot component of the link in Figure \ref{fig:link} is different. Therefore, the contact invariants $\mathrm{c}(L(a+b+2,1), \xi_{k,l})$ are in different $\mathrm{Spin}^\mathrm{c}$ structures for different $(k,l)$ pairs.
Thus, the invariants $\mathcal{\widehat{L}}(L^\prime_{k,l})$ are different.

Since $k$ and $l$ are odd with $k+l=n$ fixed, the number of $(k,l)$ pairs is $\frac{k+l}{2}$.

Let $\Diff^+(S^3, L_{k,l})$ denote the space of orientation-preserving diffeomorphisms from $S^3$ to itself that fix $L_{k,l}$ pointwise. We denote by $\Diff_0^+(S^3, L_{k,l})$ those elements in $\Diff^+(S^3, L_{k,l})$ that can be connected to the identity map through a one-parameter family of maps in $\Diff^+(S^3, L_{k,l})$. Using these, we define the \emph{mapping class group} $\mathrm{MCG}(S^3, L_{k,l})$ as
$$\mathrm{MCG}(S^3, L_{k,l})=\faktor{\Diff^+(S^3, L_{k,l})}{\Diff_0^+(S^3, L_{k,l})}.$$

$\mathcal{\widehat{L}}(L_{k,l})$ is an element of $\faktor{\mathrm{\widehat{HFK}}(-S^3,L_{k,l} )}{\mathrm{Aut}(\widehat{HFK}(-S^3,L_{k,l} ))}$, for details see \cite{LOSSz08}. According to \cite[Theorem 2.4.]{OS10} this can be lifted to $\faktor{\mathrm{\widehat{HFK}}(-S^3,L_{k,l} )}{\mathrm{MCG}(S^3, L_{k,l})}$.
Due to \cite{HT85, RW08, OS10}, we know that for non-torus 2-bridge knots $|\mathrm{MCG}(S^3, L_{k,l})|=2$. So after considering the action of the mapping class group that can map the $\mathcal{\widehat{L}}(L_{k,l})$ invariants to each other, we obtain that there are at least $\lceil \frac{k+l}{4}\rceil$ different invariants.

It is easy to check the value of the invariants $\tb(L_{k,l})$ and $\rot(L_{k,l})$. Computations can be found in the proof of Lemma \ref{lemma:HFKA=0}. Thus we are ready with the proof of Theorem \ref{sajat}.
\end{proof}

For Theorem \ref{sajattransv}, to calculate the self-linking number of a transverse knot $T$ we can use that $\sl(T)=\tb(L)-\rot(L)$ where $L$ is the Legendrian approximation of $T$.

\section{Further Legendrian non-simple two bridge knots }\label{sec:alg}

In this section, we observe what parity conditions are needed for the above proof to work. We give an algorithmically efficient method that tells from the continued fraction expansion of any rational number whether the necessary conditions hold for the corresponding two-bridge knot. This way we get further examples of Legendrian non-simple two-bridge knots. We also prove Lemma \ref{lemma:w=0} and Lemma \ref{lemma:HFKA=0}.
\\

The proof of Theorem \ref{sajat} is based on two lemmas, that is, two conditions have to hold:
\begin{itemize}
\item $\w({\widetilde{L}})=0$,
\item in Alexander grading $\tilde{A}=A(\mathcal{L}(L_{k,l}))$ the $\widehat{\HFK}$ homology of the third term $(-L(a+b+1,1), K_0)$ vanishes.
\end{itemize}

\begin{subsection}{Determining the cases when the winding number $\w({\widetilde{L}})=0$}\label{subsec:w=0}
$ $

Consider a term $\widetilde{L}_i$ of the knot $\widetilde{L}$ shown in Figure \ref{fig:Lihullam}.
$\widetilde{L}_i$ connects four points of the outer rectangle with four points of the inner one.
Denote the four points on the outer rectangle according to Figure \ref{fig:Lihullam}: the upper left by 1, the lower left by 2, the lower right by 3, the upper right by 4. Each point is connected to one of the four points on the inner rectangle by an arc, and $\widetilde{L}_i$ is the union of these arcs. Let $\pi_i(1)$, $\pi_i(2)$, $\pi_i(3)$ and $\pi_i(4)$ denote the other end of the arcs belonging to the points 1, 2, 3 and 4, respectively, on the inner rectangle.
Notice that $\pi_i(4)$ is always the upper right point of the inner rectangle, while the other three permute depending on the parity of $c_i$ and $d_i$ (the number of crossings in $\widetilde{L}_i$, see Figure \ref{fig:Lihullam}). Therefore, $\pi_i: \{1,2,3,4\}\rightarrow \{1,2,3,4\}$ is a permutation so that 4 is a fixed point.
This way we can assign an element of the symmetric group of order three ($S_3$) to $\widetilde{L}_i$.
Composing permutations $\pi_i$ for $i=1,...,m$ we get how the four points on the outermost rectangle in $\widetilde{L}$ connect to the points on the innermost one. Recall from the construction of $\widetilde{L}$ given in Section \ref{bizonyitasok} that $\widetilde{L}$ is considered in the solid torus, thus for its outermost term $\widetilde{L}_1$ the left and the right hand side of the outer rectangle are identified (i.e. on the outermost rectangle of $\widetilde{L}$ the two upper points and the two lower points are connected). The innermost term $\widetilde{L}_m$ had no further terms in its inner box, but two horizontal arcs. This means that the two upper points and the two lower points of the innermost rectangle of $\widetilde{L}$ are also connected.

Consider the presentation $S_3\cong \langle a,b | a^2=b^2=(ab)^3=\id \rangle$. This way we can represent the elements of $S_3$ by $\id, a, b, ab, ba$ and $aba$. Depending on the parity of $c_i$ and $d_i$, four permutation elements can appear as $\pi_i$ associated to $\widetilde{L}_i$. The following table shows these permutation elements determined by the parity of $c_i$ and $d_i$. The element $a$ is the swap of the lower two positions and $b$ is the swap of the two positions in the left column.

\begin{figure}[ht]
\centering
\includegraphics[width=0.8\textwidth]{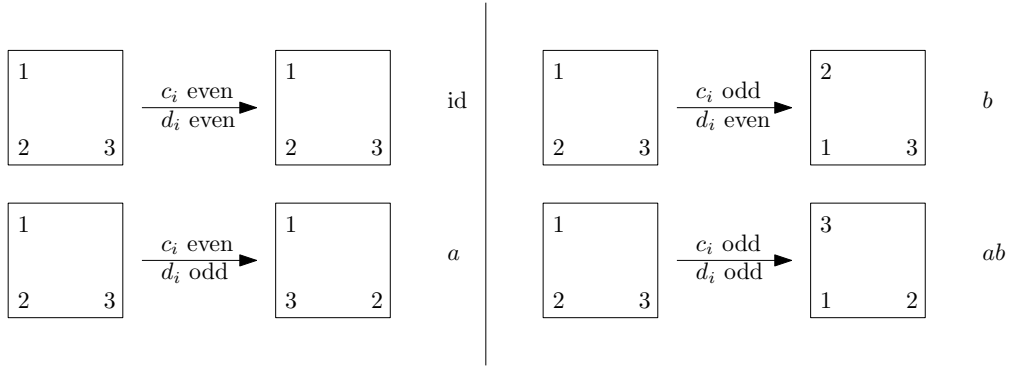}
\caption{The permutation elements determined by the parity of crossings $c_i$ and $d_i$ in $\widetilde{L}_i$}
\label{fig:permutacio}
\end{figure}

Without loss of generality we can assume that the orientation of $\widetilde{L}$ is so that for every term $\widetilde{L}_i$ the upper right arc numbered with 4 is pointing outwards from the outer rectangle of $\widetilde{L}_i$ (that is, arc 4 is pointing right in Figure \ref{fig:Lihullam}). Therefore, $\widetilde{L}$ is a 2-component link if $\pi_m\circ \pi_{m-1}\circ...\circ \pi_1 (1)$ is the upper left corner of the innermost rectangle. This happens when $\pi_m\circ \pi_{m-1}\circ...\circ \pi_1$ is the element $a$ or $\id$ in $S_3$. $\widetilde{L}$ is a knot in any other case.

The outermost term $\widetilde{L}_1$ has two arcs that leave the outer bounding box of $\widetilde{L}_1$ (see Figure \ref{fig:Lihullam}). Therefore, $\widetilde{L}$ has two arcs that go around the torus. This way the winding number of $\widetilde{L}$ is either 2 or 0, and the composition $\pi_m\circ \pi_{m-1}\circ...\circ \pi_1$ determines it:\\
If the arc in $\widetilde{L}$ which points into the outermost rectangle of $\widetilde{L}_1$ in position 1 leaves the rectangle in position 3, then $\w({\widetilde{L}})=2$.
$\w({\widetilde{L}})=0$ if and only if the arc in $\widetilde{L}$  which points into the outermost rectangle of $\widetilde{L}_1$ in position 1 leaves the rectangle in position 2, that is, if $\pi_m\circ \pi_{m-1}\circ...\circ \pi_1 (1)$ and $\pi_m\circ \pi_{m-1}\circ...\circ \pi_1(2)$ are in the lower two positions of the innermost rectangle. This happens when $\pi_m\circ \pi_{m-1}\circ...\circ \pi_1$ is $ab$ or $aba$.
\end{subsection}
\\

\begin{proof}[Proof of Lemma \ref{lemma:w=0}]
First, recall from Section \ref{bizonyitasok} that in our construction of $\widetilde{L}$ we put together $m$ pieces of $\widetilde{L}_i$ so that $a_1=k+l$, $a_{2i}=d_i$, $a_{2i+1}=c_i$ for $i=1,...,m$. Our condition that $a_2=d_1$ is odd ensures that the permutation $\pi_1$ is $a$ or $ab$ (see Figure \ref{fig:permutacio}). Since $a_{2i}=d_i$ is even for $i>1$, all $\pi_2$, $\pi_3$,... and $\pi_m$ are either 1 or $b$. This means that $\pi_m\circ \pi_{m-1}\circ...\circ \pi_1=a b^n$ where $n$ is odd because $\sum\limits_{i=1}^m a_{2i+1}$ is odd. After simplifying with $b^2=id$, we get that $\pi_m\circ \pi_{m-1}\circ...\circ \pi_1=a b$. According to the argument written in Section \ref{subsec:w=0}, this means that $\w({\widetilde{L}})=0$. 
\end{proof}


\begin{subsection}{A method to decide whether $\widehat{\HFK}(-L(a+b+1,1), K_0)=0$ in Alexander grading $\tilde{A}=A(\mathcal{L}(L_{k,l}))$}\label{subsec:HFK=0}
$ $

In this subsection we give a method to compute the $\widehat{\HFK}$ homology of the third term $(-L(a+b+1,1), K_0)$ in the distinguished triangle of knots in Figure \ref{fig:triangle} in Alexander grading $\tilde{A}=A(\mathcal{L}(L_{k,l}))$. 

From $L_{k,l}$ we can obtain $\widetilde{L}$ by removing $2a+2b+1$ half twists, which is equivalent to writing 1 instead of $a_1$ in the continued fraction expansion of the corresponding rational number.

Two-bridge knots are alternating knots, therefore their Alexander polynomial and signature are easily computable and they determine their knot Floer homology. The following theorem is known:

\begin{theorem}[Ozsváth-Szabó \cite{OSz03}]\label{thm:HFKkiszam}
Let $K$ be an alternating knot and $\Delta_K(t)=\sum\limits_{i=-n}^n a_i\cdot t^i$ the symmetrized Alexander polynomial of $K$. Then for the knot Floer homology of $K$
$$\widehat{HFK}_{i,j}(K)\cong 
\begin{cases}
0, & \mathrm{ for } \ j \neq i + \frac{\sigma}{2} \\
\mathbb{F}^{|a_i|} , & \mathrm{ for }\  j = i + \frac{\sigma}{2}
\end{cases}$$ where $i$ denotes the Alexander grading and $j$ the Maslov grading.
\end{theorem}

Now, it is enough to compute the value of the Alexander grading $\tilde{A}=A(\mathcal{L}(L_{k,l}))$. By \cite{OS10},

$$2\tilde{A}=\tb(L_{k,l})-\rot(L_{k,l})+1.$$

$2\rot(L_{k,l})$ is the difference of up-cusps and down-cusps in the diagram of $\widetilde{L}$ (Figure \ref{fig:Lihullam}), that is always 0.
However, we will see that $\tb$ depends on the parity of $c_i$'s and $d_i$'s in the terms of $\widetilde{L}$.
We know that $\tb=\mathrm{wr}-\frac{1}{2}\mathrm{\#cusps}$, where wr is the \emph{writhe}, the signed sum of crossings in the diagram and $\#$cusps denotes the number of cusps.
The number of cusps is always $2\sum\limits_{i=1}^m d_i$.

For computing the writhe we need to know the sign of the crossings in $\widetilde{L}$. We will determine these termwise for each $\widetilde{L}_i$:
\\
In $\widetilde{L}_i$ there are $c_i+d_i$ crossings. The sign of all the upper crossings ($c_i$ in Figure \ref{fig:Lihullam}) is the same, and so is for the lower ones ($d_i$ in Figure \ref{fig:Lihullam}). Note that the sign of a crossing only depends on the parity of $d_i$ and on the orientation of the arcs of the outer rectangle of $\widetilde{L}_i$. Without loss of generality we can assume that the orientation of $\widetilde{L}$ is so that for every term $\widetilde{L}_i$ the upper right arc numbered with 4 in Figure {\ref{fig:Lihullam}} is pointing outwards from the outer rectangle.

\begin{proposition}\label{prop:writhe}
The following procedure gives wr$(\widetilde{L})$.
\end{proposition}

In Figure \ref{fig:writhe} we show a diagram that helps the computation. First, we write and check the method in details, then we demonstrate the computation on an example.

\begin{figure}[ht]
\centering
\includegraphics[width=0.8\textwidth]{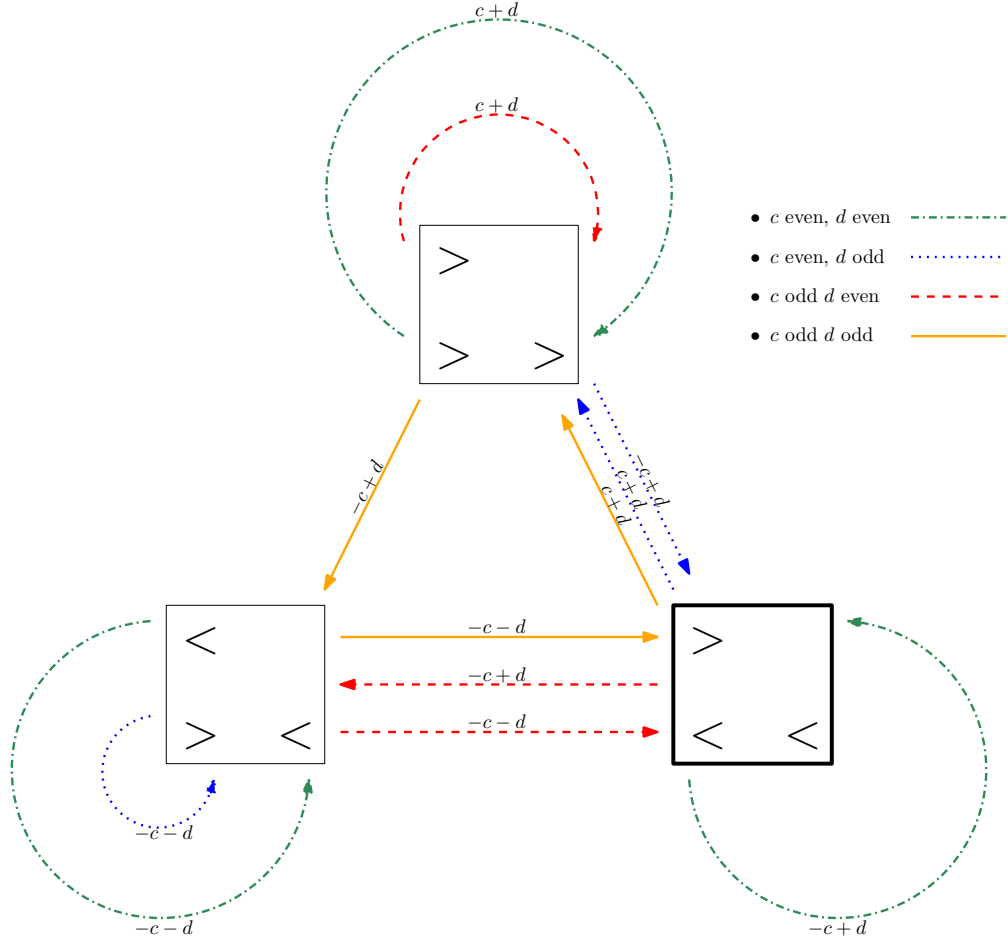}
\caption{Computation of the writhe: diagram showing the method described in Proposition \ref{prop:writhe}}
\label{fig:writhe}
\end{figure}

For a term $\widetilde{L}_i$ the signs $<$ and $>$ ("coming in" or "going out") in a box in Figure \ref{fig:writhe} show the orientation of the arcs at the outer rectangle of $\widetilde{L}_i$, that is, the arcs numbered with 1, 2, 3 and 4 in Figure \ref{fig:Lihullam}. The lower right box is marked, which means that this shows the orientations of the arcs at the outer rectangle of the first term $\widetilde{L}_1$. (We are only interested in the cases when the winding number $\w({\widetilde{L}})=0$ and since we assumed that the upper right arc always points outwards from the outer rectangle, this is the only possible orientation for $\widetilde{L}_1$.)

Four arrows start from each box, one for each parity option on $c_i$ and $d_i$. (The dashing shows the parity of $c_i$ and $d_i$.) Suppose that we know the orientation of the outer rectangle of $\widetilde{L}_i$ and the parity of $c_i$ and $d_i$. In the diagram there is a box corresponding to that orientation of the outer arcs and a unique arrow from that box corresponding to the parity. The endpoint of this arrow gives the orientation of the arcs entering the inner rectangle of $\widetilde{L}_i$ (that can be considered as the outer rectangle of $\widetilde{L}_{i+1}$). It is not hard to check that the number written on this arrow is the signed sum of crossings in $\widetilde{L}_i$. 


Starting from the lower right marked box, follow the arrow corresponding to the parity of $c_1$ and $d_1$ and consider the number written on it, that is, the signed sum of the crossings in $\widetilde{L}_1$. The endpoint of this arrow gives the orientation of the arcs entering the inner rectangle of $\widetilde{L}_1$, which is the same as the outer rectangle of $\widetilde{L}_2$. Now start again from this box (the endpoint of the previous arrow), follow the corresponding arrow according to the parity of $c_2$ and $d_2$, and add the number written on it to the previously counted signed number of crossings.
 
Then continue this procedure for every term $\widetilde{L}_i$ starting from the box which was the endpoint of the previous arrow, follow the arrow corresponding to the parity of $c_i$ and $d_i$, and add the numbers written onto it. At the end, we summed all the crossings of $\widetilde{L}$ with the proper signs, this way we got wr$(\widetilde{L})$.

\newpage
\begin{example}{Computing wr($K_\frac{153}{179}$)}\label{ex:wr}

The continued fraction expansion of the corresponding rational number is $\frac{153}{179}=[6\ 1\ 7\ 1\ 2]$. Now $\widetilde{L}$ consists of two terms: $\widetilde{L}_1$ and $\widetilde{L}_2$. The number of crossings in the first term is $c_1=a_3=7$ (odd), $d_1=a_2=1$ (odd), in the second term $c_2=a_5=2$ (even) and $d_2=a_4=1$ (odd).

Let us use the diagram in Figure \ref{fig:writhe}. We start from the lower right, marked box, since this always shows the orientation of the arcs at the outermost rectangle of the first term $\widetilde{L}_1$. Now $c_1$ and $d_1$ are both odd, so we have to follow the yellow (full) line. The number written on it is $c+d$ meaning that the signed sum of crossings in $\widetilde{L}_1$ is $c_1+d_1=7+1=8$.

The yellow (full) line points to the upper box of the diagram, this indicates the orientation of the arcs at the outer rectangle of $\widetilde{L}_2$ (which is the same as the inner rectangle of $\widetilde{L}_1$). Therefore, the next step, i.e. the calculation for the second term $\widetilde{L}_2$ starts from this upper box. Now $c_2$ is even and $d_2$ is odd, thus we follow the blue (pointed) arrow. The number on it is $-c+d$, so this orientation of the arcs implies that the signed sum of crossings in the second term $\widetilde{L}_2$ is $-c_2+d_2=-2+1=-1$.

We do not have any further terms, so the writhe is the sum of the signed sums calculated for the first two terms, that is, wr($K_\frac{153}{179})=8-1=7$.
\end{example}
\end{subsection}

\begin{proof}[Proof of Lemma \ref{lemma:HFKA=0}]
To see that the homology vanishes we give a bound on the Alexander grading $\tilde{A}=A(\mathcal{L}(L_{k,l}))$.
$$\tilde{A}=A(\mathcal{L}(L_{k,l}))=\frac{1}{2}\left(\tb(L_{k,l})-\rot(L_{k,l})+1\right).$$
Recall that $\tb=\mathrm{wr}-\frac{1}{2}\mathrm{\#cusps}$. The number of cusps in $\widetilde{L}$ is $2\sum\limits_{i=1}^m d_i$.

Now, we show that our parity assumptions imply wr$(\widetilde{L})=\sum\limits_{i=1}^m c_i+\sum\limits_{i=1}^m d_i$.
The conditions of Lemma \ref{lemma:HFKA=0} (same as of Theorem \ref{sajat}) include that in the continued fraction expansion of the corresponding rational number $a_2=d_1$ is odd and $a_{2i}=d_i$ is even for $i>1$. Using the algorithm written in Proposition \ref{prop:writhe}, $a_2=d_1$ is odd, which means that after starting at the lower right, marked box of the diagram in Figure \ref{fig:writhe} we follow either the yellow (full) or the blue (pointed) line. Notice that independently of $c_1$ the arrow points to the upper box of the diagram, and the number written on it is $c+d$ in both cases. This means that the signed sum of crossings in the first term $\widetilde{L}_1$ is $c_1+d_1$.
The condition that $a_{2i}=d_i$ is even for $i>1$ implies that starting from the top box of the diagram we always follow the green (pointed dashed) or the red (dashed) line. But independently of the $c_i$ values, these lines start and point to the same upper box of the diagram and the number written on them is $c+d$. Therefore, we got that for every term $\widetilde{L}_i$ the signed sum of crossings is $c_i+d_i$, implying wr$(\widetilde{L})=\sum\limits_{i=1}^m c_i+\sum\limits_{i=1}^m d_i$.

For the rotation number, we have $\rot(L_{k,l})=\frac{1}{2}(\mathrm{cusps}_d-\mathrm{cusps}_u)$. Our argument above shows that assuming the parity conditions of Theorem \ref{sajat}, on the diagram of Figure \ref{fig:writhe} we only go from the lower right box to the upper box. Both of these indicate an orientation of the arcs of the terms $\widetilde{L}_i$ so that there are always an equal number of up-cusps and down-cusps in every term (this number is $d_i$). Therefore, $\rot(L_{k,l})=0$.

$$\tilde{A}=A(\mathcal{L}(L_{k,l}))=\frac{1}{2}\left(\tb(L_{k,l})-\rot(L_{k,l})+1\right)=\frac{1}{2}\left(\sum\limits_{i=1}^m c_i+\sum\limits_{i=1}^m d_i-\sum\limits_{i=1}^n d_i+1\right)=\frac{1}{2}\left(\sum\limits_{i=1}^n c_i+1\right).$$
Applying Seifert's algorithm, we get a genus $g$ Seifert surface of $K_0$. Computing the Euler characteristics we get 
$$\chi=\sum\limits_{i=1}^m d_i+2-(\sum\limits_{i=1}^m c_i+\sum\limits_{i=1}^m d_i)=2-\sum\limits_{i=1}^m c_i.$$
$$g=\frac{1-\chi}{2}=\frac{\sum\limits_{i=1}^m c_i-1}{2}.$$
It is known that the Seifert genus $g_S\leq g$ bounds the largest Alexander grading with non-trivial knot Floer homology in $\widehat{HFK}(-L(a+b+1,1),K_0)$ \cite{N09}. In our case $g<\tilde{A}$, therefore $\widehat{\HFK}$ vanishes in Alexander grading $\tilde{A}=A(\mathcal{L}(L_{k,l}))$.
\end{proof}

Using the above methods we can tell from the continued fraction expansion of any rational number whether the proof of Theorem \ref{sajat} can be adapted for the corresponding two-bridge knot. This way we get further examples of Legendrian non-simple two-bridge knots. In Table \ref{tab:nullahom} we show a few of the infinitly many examples that are not included in the infinite family of Theorem \ref{sajat}. Here, we cannot always use the genus bound the way we did in the proof of Theorem \ref{sajat}. However, after using the above algorithms to compute the necessary data, we can prove that the $\widehat{\HFK}$ homology of the third term vanishes on the appropriate Alexander grading. This means that the same proof works to show they are Legendrian non-simple. We show the computations for such a knot in Example \ref{ex:nullahom}.

\begin{table}
$$\begin{array}{ccccc}
\textrm{Knot} & \textrm{Continued fraction} & \textrm{Alexander polynomial of third term}  & \tilde{A}\\
K_\frac{153}{179} & [6\ 1\ 7\ 1\ 2] & 4t^{-2} -12t^{-1} + 17-12t^{1} + 4t^2 & 3 \\
K_\frac{341}{399} & [6\ 1\ 7\ 3\ 2] & 8t^{-2} -27t^{-1} + 39 -27t^{1} + 8t^2 & 3 \\
K_\frac{529}{619} & [6\ 1\ 7\ 5\ 2] & 12t^{-2} -42t^{-1} + 61  -42t^{1} + 12t^2 & 3 \\
K_\frac{717}{839} & [6\ 1\ 7\ 7\ 2] & 16t^{-2}  -57t^{-1} + 83 -57t^{1} + 16t^2 & 3 \\
K_\frac{905}{1059} & [6\ 1\ 7\ 9\ 2] & 20t^{-2} -72t^{-1} + 105 -72t^{1} + 20t^2 & 3 \\
K_\frac{189}{221} & [6\ 1\ 9\ 1\ 2] & 5t^{-2} -15t^{-1} + 21 -15t^{1} + 5t^2 & 4 \\
K_\frac{425}{497} & [6\ 1\ 9\ 3\ 2] & 10t^{-2} -34t^{-1} + 49 -34t^{1} + 10t^2 & 4 \\
K_\frac{661}{773} & [6\ 1\ 9\ 5\ 2] & 15t^{-2} -53t^{-1} + 77 -53t^{1} + 15t^2 & 4 \\
K_\frac{897}{1049} & [6\ 1\ 9\ 7\ 2] & 20t^{-2} -72t^{-1} + 105 -72t^{1} + 20t^2 & 4 \\
K_\frac{1133}{1325} & [6\ 1\ 9\ 9\ 2] & 25t^{-2} -91t^{-1} + 133 -91t^{1} + 25t^2 & 4 \\
K_\frac{107}{127} & [6\ 2\ 1\ 5\ 1] & -3t^{-2} + 7t^{-1} -7 + 7t^{1} -3t^2 & -3 \\
K_\frac{139}{165} & [6\ 2\ 1\ 7\ 1] & -4t^{-2} + 9t^{-1} -9 + 9t^{1} -4t^2 & -4 \\
K_\frac{171}{203} & [6\ 2\ 1\ 9\ 1] & -5t^{-2} + 11t^{-1} -11 + 11t^{1} -5t^2 & -5 \\
K_\frac{293}{347} & [6\ 2\ 2\ 1\ 7] & 8t^{-2} -19t^{-1} + 23 -19t^{1} + 8t^2 & 3 \\
K_\frac{369}{437} & [6\ 2\ 2\ 1\ 9] & 10t^{-2} -24t^{-1} + 29 -24t^{1} + 10t^2 & 4 \\

\end{array}$$
\caption{Some examples of Legendrian non-simple two-bridge knots that are not included in Theorem \ref{sajat}. We can show that the $\widehat{\HFK}$ homology of the third term $(-L(a+b+1,1), K_0)$ vanishes in Alexander grading $\tilde{A}=A(\mathcal{L}(L_{k,l}))$, thus the proof of Theorem \ref{sajat} can be adapted.}
\label{tab:nullahom}
\end{table}

\begin{example}{Computations for $K_\frac{153}{179}$}\label{ex:nullahom}

We saw in Example \ref{ex:wr} that $\frac{153}{179}=[6\ 1\ 7\ 1\ 2]$, and thus $\widetilde{L}$ consists of two terms. The number of crossings is $c_1=a_3=7$, $d_1=a_2=1$, $c_2=a_5=2$ and $d_2=a_4=1$, so the condition in Theorem \ref{sajat} that $a_4$ is even fails. We can follow the proof of Theorem \ref{sajat}, but have to check that Lemma \ref{lemma:w=0} can be adapted in our case.

First, we show that $\w({\widetilde{L}})=0$.\\
Since $c_1=a_3=7$, $d_1=a_2=1$ are both odd, the permutation $\pi_1$ is $ab$ (see Figure \ref{fig:permutacio}). $c_2=a_5=2$ is even and $d_2=a_4=1$ is odd, thus $\pi_2$ is $a$. This means that $\pi_2\circ \pi_{1}=a b a$. According to the argument in (Section \ref{subsec:w=0}), this means that $\w({\widetilde{L}})=0$. 

To verify the fact that the $\widehat{\HFK}$ homology of the third term $(-L(a+b+1,1), K_0)$ in the distinguished triangle of knots (Figure \ref{fig:triangle}) vanishes in Alexander grading $\tilde{A}=A(\mathcal{L}(L_{k,l}))$ we need to compute $\tilde{A}$.
$$\tilde{A}=A(\mathcal{L}(L_{k,l}))=\frac{1}{2}\left(\tb(L_{k,l})-\rot(L_{k,l})+1\right).$$
In Example \ref{ex:wr} we checked that wr($K_\frac{153}{179})=8-1=7$. The number of cusps in $\widetilde{L}$ is $2(d_1+d_2)=2(1+1)=4$. Therefore, $\tb=\mathrm{wr}-\frac{1}{2}\mathrm{\#cusps}=7-2=5$.\\
The same argument as used in the proof of Lemma \ref{lemma:HFKA=0} shows that $\rot(L_{k,l})=0$.
$$\tilde{A}=A(\mathcal{L}(L_{k,l}))=\frac{1}{2}\left(5+1\right)=3.$$
The Alexander polynomial of $K_0$ is $4t^{-2} + -12t^{-1} + 17-12t^{1} + 4t^2$. Using Theorem \ref{thm:HFKkiszam}, we can check that in Alexander grading $\tilde{A}$ the $\widehat{\HFK}$ homology vanishes.

Note that although the homology vanishes at $\tilde{A}$, the proof of Lemma \ref{lemma:HFKA=0} does not work as simply as before in this case: after calculating $\tilde{A}$, the other computations are the same as in the proof of Lemma \ref{lemma:HFKA=0}. For the genus $g$ of a Seifert surface of $K_0$ we get that $g=\frac{1}{2}(7+2-1)=4$. But now $g \nless \tilde{A}$. Therefore, in this case it is not enough to compute this genus to see that the homology vanishes. Although the bound $g_S \leq \tilde{A}$ still holds for the Seifert genus $g_S$, we do not have a general method to compute $g_S$ directly from $\frac{p}{q}$.

\end{example}

\begin{remark}
The condition that the $\widehat{\HFK}$ homology of the third term $(-L(a+b+1,1), K_0)$ should vanish in Alexander grading $\tilde{A}=A(\mathcal{L}(L_{k,l}))$ is not necessary. Although in case of the rank of this homology group being $r>0$ the map $\widehat{F}_S$ is not an isomorphism, its kernel has rank at most $r$. Therefore, the number of distinct Legendrian realizations is at least $\lceil \frac{a_1}{4\cdot 2^r} \rceil$. Knowing this, we get even more examples of Legendrian non-simple two-brige knots. In Table \ref{tab:mashom} we give some examples of two-bridge knots such that the $\widehat{\HFK}$ homology of the third term is nontrivial, but  small enough compared to $a_1$. Thus, our previous argument can be used to prove they are Legendrian non-simple. In Proposition \ref{prop:computation} we show the computations for an example. 
\end{remark}

\begin{table}
$$ \begin{array}{cccccc}
\textrm{Knot} & \textrm{Continued fraction} & \textrm{Alexander polynomial of third term}  & \tilde{A} & r\\
K_\frac{14}{1825} & [130\ 2\ 1\ 3\ 1] & -2t^{-2} + 5t^{-1} -55t^{1} -2t^2 & -2 & 2 \\
K_\frac{243}{257} & [18\ 2\ 1\ 3\ 1] & -2t^{-2} + 5t^{-1} -5 + 5t^{1} -2t^2 & -2 & 2 \\
K_\frac{625}{661} & [18\ 2\ 1\ 3\ 3] & -2t^{-3} + 6t^{-2} -10t^{-1} + 13 -10t^{1} + 6t^{2} -2t^3 & -3 & 2 \\
K_\frac{279}{295} & [18\ 2\ 3\ 1\ 1] & -2t^{-2} + 6t^{-1} -7 + 6t^{1} -2t^2 & -2 & 2 \\
K_\frac{593}{627} & [18\ 2\ 3\ 1\ 3] & -2t^{-3} + 6t^{-2} -10t^{-1} + 13 -10t^{1} + 6t^{2} -2t^3 & -3 & 2 \\
K_\frac{677}{697} & [34\ 1\ 5\ 1\ 2] & 3t^{-2} -9t^{-1} + 13 -9t^{1} + 3t^2 & 2 & 3 \\
K_\frac{267}{275} & [34\ 2\ 1\ 1\ 1] & -t^{-2} + 3t^{-1} -3 + 3t^{1} -t^2 & -1 & 3 \\
K_\frac{601}{619} & [34\ 2\ 1\ 1\ 3] & -t^{-3} + 3t^{-2} -5t^{-1} + 7 -5t^{1} + 3t^{2} -t^3 & -2 & 3 \\
K_\frac{935}{963} & [34\ 2\ 1\ 1\ 5] & -t^{-4} + 3t^{-3} -5t^{-2} + 7t^{-1} -7 + 7t^{1} -5t^{2} + 3t^{3} -1t^4 & -3 & 3 \\
K_\frac{465}{479} & [34\ 4\ 1\ 1\ 1] & -t^{-3} + 3t^{-2} -3t^{-1} + 3 -3t^{1} + 3t^{2} -t^3 & -1 & 3 \\
K_\frac{1129}{1163} & [34\ 4\ 1\ 5\ 1] & -3t^{-3} + 7t^{-2} -7t^{-1} + 7 -7t^{1} + 7t^{2} -3t^3 & -3 & 3 \\
K_\frac{3055}{3147} & [34\ 4\ 1\ 5\ 3] & -3t^{-4} + 9t^{-3} -15t^{-2} + 19t^{-1} -19 + 19t^{1} -15t^{2} + 9t^{3} -3t^4 & -4 & 3 \\
K_\frac{1529}{1575} & [34\ 4\ 5\ 1\ 1] & -3t^{-3} + 9t^{-2} -11t^{-1} + 11 -11t^{1} + 9t^{2} -3t^3 & -3 & 3 \\
K_\frac{663}{683} & [34\ 6\ 1\ 1\ 1] & -t^{-4} + 3t^{-3} -3t^{-2} + 3t^{-1} -3 + 3t^{1} -3t^{2} + 3t^{3} -t^4 & -1 & 3 \\
\end{array}$$
\caption{Some further examples of Legendrian non-simple two-bridge knots that are not included in Theorem \ref{sajat} or in Table \ref{tab:nullahom}. Here, the $\widehat{\HFK}$ homology of the third term $(-L(a+b+1,1), K_0)$ does not vanish in Alexander grading $\tilde{A}=A(\mathcal{L}(L_{k,l}))$, but has rank $r$.}
\label{tab:mashom}
\end{table}

\begin{proof}[Proof of Proposition \ref{prop:computation}]
The corresponding rational number has the continued fraction expansion $\frac{14}{1825}=[130\ 2\ 1\ 3\ 1]$. We follow the proof of Theorem \ref{sajat} until the making of the distinguished triangle in Figure \ref{fig:triangle}.
Now the number of crossings in $\widetilde{L}$ is $c_1=a_3=1$, $d_1=a_2=2$, $c_2=a_5=1$ and $d_2=a_4=3$. Using the method written in Section \ref{subsec:w=0}, we can check that $\pi_2 \circ \pi_1$ is $bab=aba$. This means, that the winding number $\w({\widetilde{L}})=0$.

To compute the $\widehat{HFK}$ homology of the third term in the triangle we follow Section \ref{subsec:HFK=0}. We need the Alexander grading $\tilde{A}$ of the Legendrian invariant $\mathcal{L}(L_{k,l})$. According to \cite{OS10}, $2\tilde{A}=\tb(L_{k,l})-\rot(L_{k,l})+1$ and $\rot(L_{k,l})=\frac{1}{2}(\mathrm{cusps}_d-\mathrm{cusps}_u)$. In our case there are two terms: $\widetilde{L}_1$ and $\widetilde{L}_2$. In $\widetilde{L}_1$ the number of up- and down-cusps are equal. However, the orientation induces that in $\widetilde{L}_2$ there are only up-cusps which means that $\rot(L_{k,l})=-d_2=-3$.\\
$\tb=\mathrm{wr}-\frac{1}{2}\mathrm{\#cusps}$. The number of cusps is $2\sum\limits_{i=1}^2 d_i=2\cdot (2+3)=10$. Using the method written in Proposition \ref{prop:writhe} we get that $\mathrm{wr}=-c_1+d_1-c_2-d_2=-3$. Therefore, $\tb=-3-5=-8$ and $\tilde{A}=\frac{1}{2}(-8-(-3)+1)=-2$.

The Alexander polynomial of $K_0$ in Figure \ref{fig:triangle} is now $\Delta_{K_0}(t)=-2t^{-2} + 5t^{-1} -5+5t^{1} -2t^2$.
By Theorem \ref{thm:HFKkiszam}, $\widehat{HFK}_{\tilde{A},*}(K_0)\cong \mathbb{F}_2^{2}$, that is, $r=2$.

Although the map $\widehat{F}_S$ is not an isomorphism, its kernel has rank at most 2. This means that there are at least $\lceil \frac{130}{4\cdot 2^2} \rceil=9$ distinct Legendrian realizations.
\end{proof}

\bibliographystyle{plain}
\nocite*
\bibliography{biblio}

\begin{thebibliography}{10}

\bibitem{Ch02}
Yuri Chekanov.
\newblock Differential algebra of {Legendrian} links.
\newblock {\em Inventiones Mathematicae}, 150(3):441--483, Dec 2002.

\bibitem{E93}
Yakov Eliashberg.
\newblock Legendrian and transversal knots in tight contact 3-manifolds.
\newblock {\em Houston, TX: Publish or Perish, Inc.}, pages 171--193, 1993.

\bibitem{EFM01}
Judith Epstein, Dmitry Fuchs, and Maike Meyer.
\newblock {Chekanov}-{Eliashberg} invariants and transverse approximations of
  {Legendrian} knots.
\newblock {\em Pacific J. Math}, pages 89--106, 2001.

\bibitem{ENgV13}
John Etnyre, Lenhard Ng, and Vera V{\'{e}}rtesi.
\newblock {Legendrian} and transverse twist knots.
\newblock {\em Journal of the European Mathematical Society}, 15(3):969--995,
  2013.

\bibitem{EH01}
John~B. Etnyre and Ko~Honda.
\newblock Knots and contact geometry i: Torus knots and the figure eight knot.
\newblock {\em J. Symplectic Geom.}, 1(1):63--120, 12 2001.

\bibitem{HT85}
Allen {Hatcher} and William {Thurston}.
\newblock {Incompressible surfaces in 2-bridge knot complements.}
\newblock {\em Inventiones Mathematicae}, 79:225, 1985.

\bibitem{LOSSz08}
Paolo Lisca, Peter Ozsváth, András Stipsicz, and Zoltán Szabó.
\newblock {Heegard} {Floer} invariants of {Legendrian} knots in contact
  three-manifolds.
\newblock {\em Journal of the European Mathematical Society}, 11, 03 2008.

\bibitem{L98}
F.~Loose.
\newblock The tubular neighborhood theorem in contact geometry.
\newblock {\em Abhandlungen aus dem Mathematischen Seminar der Universit{\"a}t
  Hamburg}, 68(1):129--147, Dec 1998.

\bibitem{Ng05}
Lenhard Ng.
\newblock A {Legendrian} {Thurston}–{Bennequin} bound from {Khovanov}
  homology.
\newblock {\em Algebr. Geom. Topol.}, 5(4):1637--1653, 2005.

\bibitem{N09}
Yi~Ni.
\newblock Link {Floer} homology detects the {Thurston} norm.
\newblock {\em Geom. Topol.}, 13(5):2991--3019, 2009.

\bibitem{OS10}
Peter Ozsváth and András~I. Stipsicz.
\newblock Contact surgeries and the transverse invariant in knot {Floer}
  homology.
\newblock {\em Journal of the Institute of Mathematics of Jussieu},
  9(3):601–632, 2010.

\bibitem{OSz03}
Peter Ozsváth and Zoltán Szabó.
\newblock {Heegaard} {Floer} homology and alternating knots.
\newblock {\em Geom. Topol.}, 7(1):225--254, 2003.

\bibitem{OSz04}
Peter Ozsváth and Zoltán Szabó.
\newblock Holomorphic disks and knot invariants.
\newblock {\em Advances in Mathematics}, 186(1):58--116, aug 2004.

\bibitem{RW08}
Alan~W. Reid and Genevieve~S. Walsh.
\newblock Commensurability classes of 2{\textendash}bridge knot complements.
\newblock {\em Algebraic {\&} Geometric Topology}, 8(2):1031--1057, jul 2008.

\end{thebibliography}
\end{document}